\documentclass[11 pt, twoside]{article}
\usepackage[a4paper,width=150mm,top=25mm,bottom=30mm]{geometry}
\usepackage[pdftex]{pict2e}
\usepackage{amsmath,amssymb,amscd,amsthm}
\usepackage{a4wide,pstricks,xypic}
\usepackage{pdfsync}
\usepackage{fancyhdr}
\usepackage{makebox, bigstrut}
\usepackage[T1]{fontenc}
\usepackage[utf8]{inputenc} 
\usepackage[english]{babel} 
\usepackage{setspace}
\usepackage[mathscr]{euscript}
\usepackage{bm}
\usepackage{bbm}
\usepackage{tasks}
\usepackage{dsfont} 
\usepackage{amsthm}
\usepackage{amsmath}
\usepackage{amssymb}
\usepackage{mathrsfs}
\usepackage{mathtools}
\usepackage{enumitem}
\usepackage{hyperref}
\usepackage{amssymb}
\usepackage{calligra}
\usepackage{ytableau}
\usepackage{marginnote}
\usepackage{scalerel} 
\usepackage{shellesc}
\usepackage{tikz-cd}
\usepackage{faktor}
\usepackage{cleveref}
\renewcommand{\phi}{\varphi}

\let\emptyset\varnothing

\usepackage{comment}
\usepackage[all,cmtip]{xy}
\newcommand{\C}{\mathbb{C}}

\newcommand{\N}{\mathbb{N}}

\newcommand{\Q}{{\mathbb Q}}
\newcommand{\R}{\mathbb{R}}

\newcommand{\Z}{\mathbb{Z}}

\crefformat{section}{\S#2#1#3} 
\crefformat{subsection}{\S#2#1#3}
\crefformat{subsubsection}{\S#2#1#3}
\providecommand{\keywords}[1]{\textbf{\textbf{Key words---}} #1}

\theoremstyle{plain}
\numberwithin{equation}{subsection}

\let\oldmarginpar\marginpar
\renewcommand\marginpar[1]{\-\oldmarginpar[\raggedleft\footnotesize #1]
{\raggedright\footnotesize #1}}

\makeindex

\xyoption{all}
\input{xypic}

\newtheorem{teorema}{Theorem}[subsection]
\newtheorem{prop}[teorema]{Proposition}

\newtheorem{lemma}[teorema]{Lemma}

\theoremstyle{remark}
\newtheorem{oss}[teorema]{Remark}
\newtheorem{esempio}[teorema]{Example}

\theoremstyle{definition}
\newtheorem{definizione}[teorema]{Definition}

\newcounter{margin}
{\end{itshape}  \bigskip}

\DeclareMathOperator{\Hom}{Hom}

\DeclareMathOperator{\Gl}{GL}

\DeclareMathOperator{\Imm}{Imm}

\DeclareMathOperator{\PGl}{PGL}

\DeclareMathOperator{\tr}{tr}

\DeclareMathOperator{\spec}{Spec}

\DeclareMathOperator{\Stab}{Stab}
\DeclareMathOperator{\grr}{gr}

\DeclareMathOperator{\Plelog}{Log}

\begin{document}

\title{On the cohomology of character stacks \\ for non-orientable surfaces}

\author{ Tommaso Scognamiglio
\\ {\it Université Paris Cité/IMJ-PRG}
\\{\tt tommaso.scognamiglio@imj-prg.fr}
}

\maketitle
\pagestyle{myheadings}

\keywords{Character varieties, non-orientable surfaces, mixed Hodge structure \\ \textbf{Mathematics subject classification}: 14M35,14D23}

\begin{abstract} We give a counterexample to a formula suggested by the work of Letellier and Rodriguez-Villegas  \cite{NonOr} for the mixed Poincar\'e series of character stacks for non-orientable surfaces. The counterexample is obtained by an explicit description of these character stacks for (real) elliptic curves. 
\end{abstract}

\tableofcontents
\newpage

\section{Introduction}

 Let $K$ be an algebraically closed field, $r,k \geq 1$ be non-negative integers and $\mathcal{C}=(O_1,\dots,O_k)$ a $k$-tuple of semisimple orbits of $G=\Gl_n(K)$. Consider a couple $(\Tilde{\Sigma},\sigma)$ where  $\tilde{\Sigma}$ is a Riemann surface of genus $r-1$ and $\sigma: \tilde{\Sigma} \to \Tilde{\Sigma}$ an \textbf{antiholomorphic} involution without fixed points. The character stack $\mathcal{M}_{\mathcal{C}}^{\epsilon}$ associated to such a couple $(\Tilde{\Sigma},\sigma)$ and $\mathcal{C}$ is the stacky quotient \begin{equation}
    \mathcal{M}^{\epsilon}_{\mathcal{C}}\coloneqq\bigl[ \{D_1,\dots ,D_r \in G \ , \ Z_1 \in O_1,\dots ,Z_k \in O_k \ | \ D_1\theta(D_1)\cdots D_r\theta(D_r)Z_1\cdots Z_k=1  \}/G\bigr]
\end{equation}

where $\theta:G \to G$ is the Cartan involution $\theta(A)=(^t A)^{-1}$. For a more detailed definition and the relation between $\mathcal{M}_{\mathcal{C}}^{\epsilon}$ and representations of $\pi_1(\Tilde{\Sigma})$ see Section \cref{chstack}.

The stacks $\mathcal{M}_{\mathcal{C}}^{\epsilon}$ are deeply related to \textit{branes} inside the moduli space of Higgs bundles: the computation of cohomology and geometry of branes is a key part in understanding mirror symmetry for the Hitchin system. References about the subject can be found for example in \cite{ba1},\cite{ba2},\cite{GP},\cite{BGP}.

\vspace{5 pt}
Recently Letellier and Rodriguez-Villegas (see \cite[Theorem 4.6]{NonOr}) computed the E-series $E(\mathcal{M}_{\mathcal{C}}^{\epsilon},q)$ of these stacks over $\C$  when $\mathcal{C}$ is \textit{generic} (for a definition of generic $k$-tuples of orbits see Definition \ref{genericdefi}). The E-series is a specialization of the whole (compactly supported) mixed-Poincar\'e series $H_c(\mathcal{M}_{\mathcal{C}}^{\epsilon},q,t)$ obtained by plugging $t=-1$. For a definition of the mixed-Poincar\'e series see \cref{Poinca}.  E-series  give important information such as the number of irreducible components or non emptyness. 

The computation of \cite[Theorem 4.6]{NonOr} is obtained via reduction over $\mathbb{F}_q$ and point counting. The authors consider the $\mathbb{F}_q-$ stack $\mathcal{M}_{\mathcal{C},\mathbb{F}_q}^{\epsilon}$ and compute in an explicit way a rational function $Q(t)$ such that \begin{equation}
    Q(q^n)=\# \mathcal{M}_{\mathcal{C},\mathbb{F}_q}^{\epsilon}(\mathbb{F}_{q^n})
\end{equation}  
 In this case  there is an equality $E(\mathcal{M}_{\mathcal{C},\C}^{\epsilon},q)=Q(q)$, as shown for example in \cite[Theorem 2.8]{NonOr}. Surprisingly, the functions $Q(t)$ appearing in this context are very similar to the ones computing E-series of character stacks for Riemann surfaces.
 
 Consider $g \geq 0$, $k \geq 1$ and $\mathcal{C}=(O_1,\dots O_k)$ as above. The associated character stack $\mathcal{M}_{\mathcal{C}}$ for a Riemann surface $\Sigma$ of genus $g$ is the stacky quotient $$\mathcal{M}_{\mathcal{C}}\coloneqq\bigl[\{A_1,B_1, \dots , A_g,B_g \in G \ , \ X_1 \in O_1 , \dots ,X_k \in O_k \ |  [A_1,B_1]\cdots [A_g,B_g]X_1 \cdots X_k=1\}/G \bigr] .$$
 This can alternatively be described as the quotient stack  $$\bigl[\{\rho: \pi_1(\Sigma-\{x_1,\dots,x_k\}) \to G \ | \  \rho(y_i) \in O_i \}/G \bigr]$$ where $\{x_1,\dots ,x_k\} $ is a set of $k$ points of $\Sigma$ and each $y_i$ is a small loop around $x_i$. E-series for these stacks and \textit{generic} orbits were computed in \cite[Theorem 1.2.3]{HLRV}. As observed in \cite[Remark 1.5]{NonOr}, for $r=2h$  we have an equality $E(\mathcal{M}_{\mathcal{C}}^{\epsilon},q)=E(\mathcal{M}_{\mathcal{C}},q)$ where $\mathcal{M}_{\mathcal{C}}$ is associated to a Riemann surface of genus $h$. Even for $r$ odd, the formulas for the E-series of $\mathcal{M}_{\mathcal{C}}^{\epsilon}$  are very similar to those of $E(\mathcal{M}_{\mathcal{C}},q)$  (see \cref{EP} for more details).

There is a longstanding conjecture about the whole mixed Poincar\'e series of character stacks for Riemann surfaces \cite[Conjecture 1.2.1]{HLRV}: in \cite[Theorem 4.8]{NonOr} the authors verified that a completely analogous formula holds for $\mathcal{M}_{\mathcal{C}}^{\epsilon}$ for $r=1$ and $k=1$ (see \cref{EP} for more details). It would therefore have been natural to expect a similar formula to hold for all $r$.  The main result of this paper (see \cref{mainresult}) is an explicit description of some of these spaces and their cohomology in the case  $r=2$ giving a counterexample to the expected formula. The main theorem is:

\begin{teorema}
\label{mainteo}
Put $r=2$, $k=1$ and consider the orbit $\{e^{\frac{\pi d i}{n}}\}$ where $(n,d)=1$ and $d$ is even. Then $\mathcal{M}^{\epsilon}_{\mathcal{C}}$ is a $\mu_2$-gerbe over $\C^*$. In particular, its mixed Poincar\'e series is $$H_c(\mathcal{M}_{\mathcal{C}}^{\epsilon},q,t)=qt^2+t $$ 
\end{teorema}

 To prove Theorem \ref{mainteo} we need some results of independent interest concerning the geometry of the spaces $\mathcal{M}_{\mathcal{C}}^{\epsilon}$ for $k=1$ and the orbit $\{e^{\frac{\pi d i}{n}}\}$ (see \cref{geomd}).
 
 To summarize these results, let $M_{n,d}^{\epsilon}$ be the GIT quotient associated to $\mathcal{M}_{\mathcal{C}}^{\epsilon}$ and $M_{n,d}$ be the GIT quotient of the character stack associated to the Riemann surface $\Tilde{\Sigma}$ (of genus $r-1$) for $k=1$ and the orbit $\{e^{\frac{2 \pi di}{n}}\}$. There is an involution on $M_{n,d}$, which we denote again by $\sigma$,  which sends a representation $\tilde{\rho} \in M_{n,d}$ to $\sigma(\tilde{\rho})=\theta(\tilde{\rho})(\sigma_*)$ (for more details and a definition of $\sigma_{*}$ see \cref{prel1},\cref{chstack}). In \cref{geomd} we show that:

\begin{teorema}
\label{fattapr}
If $r$ is odd, the fixed point locus $M_{n,d}^{\sigma}$ is isomorphic to $M_{n,d}^{\epsilon}$. If $r$ is even, there is an open-closed decomposition $M_{n,d}^{\sigma}=M_{n,d}^{\sigma,+} \bigsqcup M_{n,d}^{\sigma,-}$ such that  $M_{n,d}^{\sigma,+} \cong M_{n,d}^{\sigma,-} \cong M_{n,d}^{\epsilon}$.
\end{teorema}

The Theorem \ref{fattapr} (and the others in \cref{geomd}) are probably known to the experts  but we could not locate a reference in the literature. We review them here for the sake of completeness. In the paragraph \cref{mainresult}, we describe the variety $M_{n,d}^{\sigma}$ for $r=2$ and find an isomorphism $M_{n,d}^{\sigma,+} \cong \C^*$ which allows to prove Theorem \ref{mainteo}.

\paragraph{Acknowledgements.}

I would like to thank Emmanuel Letellier for bringing this subject to my attention. It is a pleasure to thank also Florent Schaffhauser for many useful discussions about the topics dealt in this paper.

\section{Preliminaries}

\subsection{Fundamental groups of punctured non-orientable surfaces}
\label{prel1}
Let $\Tilde{\Sigma}$ be a Riemann surface of genus $g$ and $\sigma: \Tilde{\Sigma} \to \Tilde{\Sigma}$ be an antiholomorphic involution $\sigma$ without fixed points. The quotient $\Sigma \coloneqq \Tilde{\Sigma}/\sigma$ is endowed with the structure of a non-orientable surface: topologically $\Sigma$ is the connected sum of $r \coloneqq g+1$ real projective planes. We denote by $p$ the quotient map $p:\Tilde{\Sigma} \longrightarrow \Sigma .$ 

Let $S=\{z_1,\dots ,z_k\}$ be a set of $k$ points on $\Sigma$. We fix also a basepoint $x_0 \in \Sigma - S$ and a point $\tilde{x_0}$ in the fiber  $p^{-1}(x_0)$. We denote the fundamental groups of $\Sigma-S,\tilde{\Sigma}-p^{-1}(S)$ with basepoints $x_0,\tilde{x}_0$ by $\pi_1(\Sigma - S)$ and $\pi_1(\Tilde{\Sigma}-p^{-1}(S))$ respectively. The map $p$ induces an injective morphism $p_*:\pi_1(\Tilde{\Sigma}-p^{-1}(S)) \to \pi_1(\Sigma - S) .$   We fix also a path $\lambda_{\sigma}:\tilde{x_0} \to \sigma({\tilde{x_0}})$ inside $\Tilde{\Sigma}-p^{-1}(S)$: its projection determines a closed path $p(\lambda_{\sigma}) \in \pi_1(\Sigma - S) .$

Finally, we denote by $\sigma_*$ the morphism on $\pi_1(\tilde{\Sigma}-p^{-1}(S))$ given by $$\sigma_*(c)=\lambda_{\sigma}^{-1}\sigma(c)\lambda_{\sigma} .$$

Notice that the morphism $\sigma_*$ is not an involution in general. More precisely,  $\sigma_*^2$ is the conjugation by the element $\sigma(\lambda_{\sigma})\lambda_{\sigma} \in \pi_1(\Tilde{\Sigma}-p^{-1}(S))$ which in general is not the identity.
We recall that there are explicit presentations of the above fundamental groups: \begin{equation}\label{present1}\Pi\coloneqq\pi_1(\Sigma-S)=\langle d_1^2 \cdots d_r^2 z_1 \cdots z_k=1\rangle \end{equation} and \begin{equation}\label{preent2} \tilde{\Pi}\coloneqq \pi_1(\Tilde{\Sigma}-p^{-1}(S))=\langle[a_1,b_1]\cdots[a_g,b_g]x_1 \cdots x_{2k}=1\rangle .\end{equation}

\vspace{10 pt}

\begin{esempio}
\label{ellip}
 Let us look at the case of $r=2$ and $k=1$. We consider the elliptic curve $\tilde{E}$ associated to the lattice $\left<1,i\right>\subseteq \C$ i.e $\tilde{E} \cong \C/\left<1,i\right> $ and let $\pi$ be the projection $\pi: \C \to E$. Let $\sigma:\Tilde{E} \rightarrow \Tilde{E} $ be the involution without fixed points defined by $ \sigma(z)= \bar{z}+\frac{1}{2} $ and $p: \Tilde{E } \to E \coloneqq \Tilde{E}/\sigma $ be the associated quotient.
 
 We fix a point $z_1 \in E$ and we let its preimage in $\Tilde{E}$ be   $p^{-1}(z_1)=\{y_1,y_2\}$. Put $x_0=p(0)$ and $\tilde{x_0}=0$ as base points and $\lambda_{\sigma}=\pi(\gamma(t))$  where $\gamma(t)=\dfrac{1}{2}t .$ Denoting by $a,b$ the paths $a(t)=\pi(it) $ and $b(t)=\pi(t) $, the fundamental group $\pi_1(\Tilde{E}-\{y_1,y_2\})$ admits the presentation \begin{equation}
    \label{rel}
   \langle b^{-1}a^{-1}ba  =x_2x_1\rangle.
\end{equation} 
where  $x_1,x_2$ are two loops around  $y_1,y_2$. It is not difficult to compute that \begin{equation}\label{eqell1} \sigma_*(a)=x_1a^{-1}\end{equation} and \begin{equation}
    \label{eqell2}\sigma_*(b)=b
\end{equation} Moreover, the following equalities hold: $\lambda_{\sigma}^{-1}\sigma(x_1)\lambda_{\sigma}=ax_1^{-1}x_2^{-1}x_1a^{-1} $ and $ \lambda_{\sigma}^{-1}\sigma(x_2)\lambda_{\sigma}=bax_1^{-1}a^{-1}b^{-1} .$

\end{esempio}

\subsection{Character stacks for non-orientable surfaces}
\label{chstack}

We fix an algebraically closed field $K$ (which for us will be either $\C$ or $\overline{\mathbb{F}}_q$). We denote by $G$ the general linear group $\Gl_n(K)$ and by $\theta$ the Cartan involution $g \to (^t g)^{-1}$. The corresponding semidirect product will be denoted by $G^+=G \rtimes \left< \theta \right>$. Let $\mathcal{C}=(O_1, \dots, O_k)$ be a $k$-tuple of semisimple orbits of $G$. We consider the variety $$\Hom^{\epsilon}_{\mathcal{C}}(\Pi,G^+)\coloneqq\{\rho: \pi_1(\Sigma-S) \to G^+ \ | \ p(\rho(d_j))=\theta \ \text{and} \ \rho(z_i) \in h(C_i) \  \text{for all} \ i,j\} $$
where $p:G\to \left<\theta \right>$ is the natural projection and $h: G \to G^+ $ the natural inclusion. Given the explicit presentation of $\pi_1(\Sigma-S)$ we can rewrite  $\Hom^{\epsilon}_{\mathcal{C}}(\Pi,G^+)$ as 
$$\Hom^{\epsilon}_{\mathcal{C}}(\Pi,G^+)=\{D_1, \dots,D_r \in G \ \text{and} \ Z_1 \in O_1 ,\dots ,Z_k \in O_k \ | \ D_1\theta(D_1)\cdots D_r\theta(D_r)Z_1\cdots Z_k=1  \} .$$

 The variety $\Hom^{\epsilon}_{\mathcal{C}}(\Pi,G^+)$ is endowed with a $G$-action defined by: \begin{equation}
    \label{action1}
    g \cdot D_i=gD_i\,{^t}g \hspace{10 pt} g\cdot Z_i=gZ_ig^{-1}.
\end{equation}

The  character stacks we will consider are the quotient stacks $$\mathcal{M}_{\mathcal{C}}^{\epsilon}=[\Hom^{\epsilon}_{\mathcal{C}}(\Pi,G^+)/G] .$$

As $\Hom^{\epsilon}_{\mathcal{C}}(\Pi,G^+)$ is affine and $G$ is reductive, we can also consider the GIT quotient $M_{\mathcal{C}}^{\epsilon}\coloneqq \Hom_{\mathcal{C}}^{\epsilon}(\Pi,G^+)/\!/G$ and the universal map $q:\mathcal{M}_{\mathcal{C}}^{\epsilon} \to M_{\mathcal{C}}^{\epsilon} .$

\vspace{10 pt}

The stacks $\mathcal{M}^{\epsilon}_{\mathcal{C}}$ admit an alternative description in terms of the so-called real $\sigma$-invariant representations (which can be found in \cite[Section 2]{cheng}\cite[Section 3, 3.2]{FS} and \cite[Remark 4.2]{NonOr}). A representation $\rho \in \Hom^{\epsilon}_{\mathcal{C}}(\Pi,G^+)$ gives by restriction a representation $\tilde{\rho}: \Tilde{\Pi} \to G$ such that the following diagram commutes \begin{equation}
\xymatrix{1\ar[r] 
&\tilde{\Pi}\ar[r]^{p_*}\ar[d]^{\tilde{\rho}}
&\Pi\ar[r]^-\chi\ar[d]^{\rho}&\langle
  \theta\rangle\ar[r]\ar[d]^{{\rm Id}}&1\\ 
1\ar[r]& G\ar[r]^\iota& G^+\ar[r]^\pi
&\langle \theta\rangle\ar[r]&1} 
\label{diag}\end{equation}

 It is therefore natural to ask conversely which representations $\tilde{\rho}$ of $\Tilde{\Pi}$ can be lifted to a morphism $\rho:\Pi \to G^+$ which makes the diagram (\ref{diag}) commute. To answer to the question, it is necessary to precisely describe monodromies around the punctures, as explained in \cite[Remark 4.2]{NonOr}. 

Let $p^{-1}(S)=\{y_{1,1},\dots ,y_{k,1},y_{1,2},\dots y_{k,2}\}$ where $\sigma(y_{i,1})=y_{i,2}$ for $i=1,\dots,k$.
We can rewrite  the standard presentation \ref{preent2} of  $\Tilde{\Pi}$ as $$\langle[a_1,b_1]\cdots [a_g,b_g]x_{1,1}\cdots x_{k,1}x_{1,2}\cdots x_{k,2}=1 \rangle . $$ where each $x_{i,j}$ is a path around $y_{i,j}$. Let $\Tilde{\mathcal{C}}$ be the  $2k$-tuple $\Tilde{\mathcal{C}}=(O_1,\dots ,O_k,O_1,\dots O_k)$ and  $\Hom_{\tilde{\mathcal{C}}}(\tilde{\Pi},G)$ be the affine variety $$\Hom_{\mathcal{\Tilde{C}}}(\tilde{\Pi},G)\coloneqq\{\tilde{\rho}:\Tilde{\Pi} \to G \ | \ \tilde{\rho}(x_{i,1}) \in O_i \ \text{and} \ \tilde{\rho}(x_{i,2}) \in O_i \} .$$

For a representation $\tilde{\rho} \in \Hom_{\mathcal{\Tilde{C}}}(\tilde{\Pi},G)$ we say that $\tilde{\rho}$ is $\sigma$-invariant if  $\tilde{\rho} \cong \theta \tilde{\rho}(\sigma_*) .$ This is equivalent to asking  for the existence of  an element $h_{\sigma} \in G$ which verifies \begin{equation}\label{1} h_{\sigma} \tilde{\rho} h_{\sigma}^{-1}=\theta \tilde\rho(\sigma_*). \end{equation} 

\begin{definizione}
Given a $\sigma$-invariant $\tilde{\rho} \in \Hom_{\mathcal{\Tilde{C}}}(\tilde{\Pi},G)$, we say that the representation $\tilde{\rho}$ is \textit{real} if there exists $h_{\sigma}$ as in eq.(\ref{1}) such that \begin{equation}\label{2}\tilde{\rho}( \sigma(\lambda_{\sigma})\lambda_{\sigma}) =h_{\sigma}^{-1}\theta(h_{\sigma}^{-1}). \end{equation}

We say that $\Tilde{\rho}$ is \textit{quaternionic} if there exists $h_{\sigma}$ as in eq.(\ref{1}) such that $\tilde{\rho}( \sigma(\lambda_{\sigma})\lambda_{\sigma}) =-h_{\sigma}^{-1}\theta(h_{\sigma}^{-1})$.
\end{definizione}

If the conditions of Equations (\ref{1}),(\ref{2}) are satisfied, the couple $(\tilde{\rho},h_{\sigma})$ can be extended to a map $\rho \in \Hom^{\epsilon}_{\mathcal{C}}(\Pi,G^+)$ such that the diagram (\ref{diag}) commutes. Let $\Tilde{\mathcal{U}}_{\mathcal{C}}$ be the variety \begin{equation}
   \label{3} 
 \Tilde{\mathcal{U}}_{\mathcal{C}}=\{(\tilde{\rho},h_{\sigma}) \in \Hom_{\tilde{\mathcal{C}}}(\tilde{\Pi},G) \times G \ \text{which verify} \ \text{Equations} \ \ref{1},\ref{2} \} .\end{equation} The variety $\Tilde{\mathcal{U}}_{\mathcal{C}} $ is endowed with a $G$-action defined by \begin{equation}
\label{action}g \cdot (\tilde{\rho},h_{\sigma})=(g\tilde{\rho}g^{-1},\theta(g)hg^{-1}) .\end{equation} The arguments above imply the following Proposition

\begin{prop}
There is an isomorphism of quotient stacks $$\mathcal{M}^{\epsilon}_{\mathcal{C}} \cong [\Tilde{\mathcal{U}}_{\mathcal{C}}/G] .$$ 
\end{prop}

\vspace{10 pt}

\begin{oss}
\label{rqt}If $\tilde{\rho}$ is an irreducible representation and $h$ is such that there is an equality $h \tilde{\rho}h^{-1}=\theta \tilde{\rho}(\sigma_*)$,
then either $h^{-1}\theta(h^{-1})= \tilde{\rho}(\sigma(\lambda_{\sigma})\lambda_{\sigma})$ or  $h^{-1}\theta(h^{-1})=-\tilde{\rho}(\sigma(\lambda_{\sigma})\lambda_{\sigma})$ and only one of the two is true (see \cite[III.5.1.2]{cheng2}), i.e an irreducible $\sigma$-invariant representation is either real or quaternionic.

\end{oss}

\vspace{8 pt}

\begin{oss}
\label{map}
It is natural to consider  the stack $\mathcal{M}_{\tilde{\mathcal{C}}}\coloneqq[\Hom_{\tilde{\mathcal{C}}}(\tilde{\Pi},G)/G]$ and the associated GIT quotient $M_{\tilde{\mathcal{C}}}$. The stacks $\mathcal{M}_{\tilde{\mathcal{C}}}$ and $M_{\tilde{\mathcal{C}}}$ admit an involution, which we denote again by $\sigma$,  induced by the map $$\sigma( \tilde{\rho})\coloneqq \theta\tilde{\rho}(\sigma_*) .$$ 

  We can define a morphism $f:M^{\epsilon}_{\mathcal{C}} \to M_{\tilde{\mathcal{C}}}^{\sigma} $  which maps a couple $(\tilde{\rho},h) $ as in Equation (\ref{3}) to the representation $\tilde{\rho}$. In a slightly more involved way, it would be possible to lift the map $f$ to a morphism of quotient stacks $F:\mathcal{M}^{\epsilon}_{\mathcal{C}} \to \mathcal{M}^{\sigma}_{\tilde{\mathcal{C}}}$. These morphisms are in general not even surjective: we will describe the image of $f$ in certain cases in Proposition \ref{connc}.

\end{oss}

\subsection{Cohomology computation}

In this paragraph we will briefly review the results obtained by Letellier and Rodriguez-Villegas in \cite{NonOr} about the character stacks $\mathcal{M}_{\mathcal{C}}^{\epsilon}$. Let us first recall the definition of the E-series and the mixed Poincar\'e series of an algebraic stack and the combinatorics needed for the formulas for the E-series $E(\mathcal{M}_{\mathcal{C}}^{\epsilon},q)$.

\subsubsection{Mixed Poincar\'e series}
\label{Poinca}
Let $\mathcal{X}$ be an algebraic stack of finite type over an algebraically closed field $k$. For $K=\C$, we will  consider the compactly-supported  cohomology groups  $ H^{*}_c(\mathcal{X})\coloneqq H_c^*(\mathcal{X},\C)$ with coefficients in $\C$. For $K=\overline{\mathbb{F}}_q$, we will denote by $H^*_c(\mathcal{X})$ the compactly supported \'etale cohomology with coefficients in $\overline{\Q}_{\ell}$.

When $K=\C$, each vector space $H^k_c(\mathcal{X})$ is endowed with the weight filtration $W^k_{\bullet}$ (see \cite[Chapter 8]{HodgeIII} for a definition  and \cite[Section 2.2]{NonOr} for the analogous one for stacks $\mathcal{X}$ over $\overline{\mathbb{F}}_q$). We define the mixed-Poincar\'e series $H_c(\mathcal{X},q,t)$ as  \begin{equation}
    \label{poinc}
    H_c(\mathcal{X};q,t)\coloneqq \sum_{k,m}\dim(W^k_m/W^k_{m-1})q^{\frac{m}{2}}t^k.
\end{equation}
The specialization  $H_c(\mathcal{X},1,t)$ of $H_c(\mathcal{X},q,t)$ at $q=1$ is equal to  the Poincar\'e series of the stack $\mathcal{X}$. When $\displaystyle \sum_{k}(-1)^k \dim(W^k_m/W^k_{m-1})$ is finite for each $m$, we define the E-series: \begin{equation}
    \label{epol}
    E(\mathcal{X},q) \coloneqq H_c(\mathcal{X};q,-1)=\sum_{m,k}\dim(W^k_m/W^k_{m-1})(-1)^kq^{\frac{m}{2}}.
\end{equation}

For a quotient stack $\mathcal{X}=[X/G]$ where $G$ is a connected linear algebraic group and $X$ an affine variety, the E-series $E(\mathcal{X},q)$ is well defined and $E(\mathcal{X},q)=E(X,q)E(BG,q)$ where $BG$ is the classifying stack of $G$: for a proof see \cite[Theorem 2.5]{NonOr}.

\subsubsection{Combinatorics}
\label{combinatorics}

We fix  integers $m,k \geq 0$ and we denote by $\mathcal{P}$ the set of partitions. Let $\mathbf{x_1}=\{x_{1,1},x_{1,2} \dots \},\dots$ $ \dots, \mathbf{x_k}=\{x_{k,1},\dots\}$ be $k$ sets of infinitely many variables and let us denote by $\Lambda_k=\Lambda(\mathbf{x_1},\dots ,\mathbf{x_k})$ the ring of functions separately symmetric in each set of variables. On $\Lambda_k$ there is a natural bilinear form obtained by  extending by linearity $$\left < f_1(\mathbf{x_1})\cdots f_k(\mathbf{x_k}),g_1(\mathbf{x_1})\cdots g_k(\mathbf{x_k})\right>=\prod_{i=1}^k \left<f_i,g_i\right> $$ where $\left<,\right>$ is the  bilinear form making the Schur functions $s_{\mu}$ an orthonormal basis. For a multipartition $\bm \mu=(\mu_1,\dots ,\mu_k) \in \mathcal{P}^k$ we denote by $h_{\bm{\mu}}=h_{\mu_1}(\mathbf{x_1})\cdots h_{\mu_k}(\mathbf{x_k})$ the associated complete symmetric function and similarly $m_{\bm{\mu}}=m_{\mu_1}(\mathbf{x_1})\cdots m_{\mu_k}(\mathbf{x_k})$.

We consider the hook functions 
\begin{equation}
 \label{defH}
 \mathcal{H}_{m,\lambda}(z,w)=\prod_{s \in \lambda} \dfrac{(z^{2a(s)+1}-w^{2l(s)+1})^m}{(z^{2a(s)+2}-w^{2l(s)})(z^{2a(s)}-w^{2l(s)+2})}
\end{equation}

and the associated series \begin{equation}
    \label{omega}
    \Omega_m(z,w)=\sum_{\lambda \in \mathcal{P}}\mathcal{H}_{m,\lambda}(z,w) \prod_{i=1}^k H_{\lambda}(\mathbf{x_i},z^2,w^2)
\end{equation}

where $H_{\lambda}(\mathbf{x_i},q,t)$ are the (modified) Macdonald symmetric polynomials (for a definition see \cite[I.11]{garsia-haiman}). We define the functions $\mathbb{H}_{\bm \mu,m}(z,w)$ by the following formula:
\begin{equation}
    \label{formula}
    \mathbb{H}_{\bm \mu,m}(z,w)\coloneqq(z^2-1)(1-w^2)\left<\Plelog(\Omega_m(z,w)),h_{\bm \mu}\right>
\end{equation}
where $\Plelog$ is the plethystic logarithm (for a definition see for example \cite[Section 2.3.3]{HLRV}).

\subsubsection{E-polynomials and conjectures}
\label{EP}
Let us now explain one of the main results of \cite{NonOr} about the stacks $\mathcal{M}_{\mathcal{C}}^{\epsilon}$. Let $\mathcal{C}=(O_1,\dots, O_k)$ be a $k$-tuple of semisimple orbits of $G$.

\begin{definizione}
\label{genericdefi}
The $k$-tuple  $\mathcal{C}$ is said to be \textit{generic} if the following property holds. Given a subspace $V$ of  $K^n$ which is stabilized by some $X_i \in O_i$, for each $i=1,\dots,k$,  such that $\displaystyle\prod_{i=1}^k \det(X_i|_V)=1  $ then either $V=\{0\}$ or $V=K^n$.
\end{definizione}
In \cite[Theorem 4.6]{NonOr} the authors showed that for  a generic $\mathcal{C}$, the following equality holds: \begin{equation}
    E(\mathcal{M}^{\epsilon}_{\mathcal{C}},q)=\dfrac{q^{\frac{d_{\mu}}{2}}}{q-1}\mathbb{H}_{\bm \mu,r}\left(\sqrt{q},\dfrac{1}{\sqrt{q}} \right)
\end{equation}
where $\bm \mu=(\mu_1,\dots ,\mu_k)$ is the  multipartition given by the multiplicities of the  eigenvalues of $O_1,\dots ,O_k$ respectively and $$d_{\bm \mu}=n^2(r-2+k)+2-\sum_{i,j}(\mu_i^j)^2 .$$ This result is surprinsigly similar to the analogous one obtained in \cite[Theorem 1.2.3]{HLRV} about character stacks for Riemann surfaces. Fix a Riemann surface $\overline{\Sigma}$ of genus $g$  and a set of $k$-points $\overline{S}=\{y_1,\dots ,y_k\} \subseteq \overline{\Sigma}$. The associated character stack $\mathcal{M}_{\mathcal{C}}$ is the quotient stack $$\mathcal{M}_{\mathcal{C}}\coloneqq\bigl[\{A_1,B_1 ,\dots ,A_g,B_g \in G,\ X_1 \in O_1 ,\dots X_k \in O_k \ | \ [A_1,B_1]\cdots [A_g,B_g]X_1\cdots X_k=1\}/G\bigr] .$$ In \cite[Theorem 1.2.3]{HLRV}  the authors showed that the following equality holds: \begin{equation}
\label{epolor}E(\mathcal{M}_{\mathcal{C}},q)=\dfrac{q^{\frac{d_{\mu}}{2}}}{q-1}\mathbb{H}_{\bm \mu,2g}\left(\sqrt{q},\dfrac{1}{\sqrt{q}} \right) \end{equation} where  $d_{\bm \mu}=n^2(2g-2+k)+2-\sum_{i,j}(\mu_{i}^j)^2$. Notice that for $r=2h$ the E-polynomial of $\mathcal{M}_{\mathcal{C}}^{\epsilon}$ agrees thus with the one of $\mathcal{M}_{\mathcal{C}}$ for a Riemann surface $\overline{\Sigma}$ of genus $h$.

In the same paper \cite[Conjecture 1.2.1]{HLRV}, the authors were also able to give a conjectural formula for the whole mixed Poincar\'e series of the stacks $\mathcal{M}_{\mathcal{C}}$,  naturally deforming Equation (\ref{epolor}).
The conjectural identity for the mixed Poincat\'e series of $\mathcal{M}_{\mathcal{C}}$ is \begin{equation}
\label{mhp}
    H_c(\mathcal{M}_{\mathcal{C}},q,t)=\dfrac{(qt^2 )^{\frac{d_{\mu}}{2}}}{qt^2-1}\mathbb{H}_{\bm \mu,2g}\left(t\sqrt{q},-\dfrac{1}{\sqrt{q}}, \right).
\end{equation}

The conjectural identity (\ref{mhp}) is generally believed to be true. The Poincar\'e series of the stacks $\mathcal{M}_{\mathcal{C}}$ were computed by Mellit in  \cite[Theorem 7.12]{Mellit2} and his result agrees with the specialization of Formula (\ref{mhp}) at $q=1$. In \cite[Theorem 4.6]{NonOr} it is proved that a formula  analogous to Formula (\ref{mhp}) holds in the non-orientable setting for $r=k=1$ i.e that the following equality holds \begin{equation}
\label{falseconj1}
    H_c(\mathcal{M}^{\epsilon}_{\mathcal{C}},q,t)=\dfrac{(qt^2)^{\frac{d_{\mu}}{2}}}{qt^2-1}\mathbb{H}_{\bm \mu,1}\left(t\sqrt{q},-\dfrac{1}{\sqrt{q}} \right).
\end{equation}
It would therefore have been natural to expect that such a formula holds for all $r,k$ i.e that \begin{equation}
    \label{falseconj}
    H_c(\mathcal{M}_{\mathcal{C}}^{\epsilon},q,t)=\dfrac{(qt^2)^{\frac{d_{\mu}}{2}}}{qt^2-1}\mathbb{H}_{\bm \mu,r}\left(t\sqrt{q},-\dfrac{1}{\sqrt{q}} \right).
\end{equation}
 for a generic $\mathcal{C}$.
The main result of this paper is a counterexample to Formula (\ref{falseconj}), obtained by an explicit description of these spaces in the case $r=2$.

\section{Main results}

\subsection{Character stacks for $k=1$ and generic orbit}
\label{geomd}

In this section we assume that $K=\C$. We fix $r \geq 1$ and a Riemann surface $\Tilde{\Sigma}$ of genus $g \coloneqq r-1$ with an antiholomorphic involution $\sigma: \Tilde{\Sigma} \to \Tilde{\Sigma}$. Consider a point $z_1 \in \Sigma=\Tilde{\Sigma}/\sigma$ and the subset $S\coloneqq\{z_1\} \subseteq \Sigma$ (i.e $k=1$). Let $d,n \in \N$ such that $d$ is even and $(d,n)=1$. Let $\mathcal{C}$ be the generic semisimple orbit of $\Gl_n(\C)$ given by $\mathcal{C}=\{e^{\pi i \frac{d}{n}}I_n \}.$ 

We denote the associated character stacks in this case by $\mathcal{M}^{\epsilon}_{n,d}\coloneqq \mathcal{M}_{\mathcal{C}}^{\epsilon}$ and $\mathcal{M}_{n,d}\coloneqq \mathcal{M}_{\Tilde{\mathcal{C}}}$ respectively and similarly  the associated GIT quotients by $M_{n,d}^{\epsilon}$ and $ M_{n,d}$ respectively.

As $\mathcal{C}$ is a central orbit, the character stack $\mathcal{M}_{n,d}$ is the twisted character stack $$\bigl[\{A_1,B_1 ,\dots ,A_g,B_g \in \Gl_n \ | \ [A_1,B_1]\cdots [A_g,B_g]=e^{\frac{2 \pi i d}{n}} \}/\Gl_n\bigr] .$$ considered for example in \cite{HRV}. As $d$ and $n$ are coprime, the representations $\tilde{\rho} \in \mathcal{M}_{n,d}$ are irreducible (as shown in \cite[Lemma 2.2.6]{HRV}). In this case, given an element $\rho \in \Hom^{\epsilon}_{\mathcal{C}}(\Pi,G^+)$ corresponding to a couple $(\tilde{\rho},h_{\sigma})$ with $\tilde{\rho} \in \Hom_{\tilde{\mathcal{C}}}(\tilde{\Pi},G)$ we have $\Stab_G(\rho)=\pm 1$ (see \cite[III.5.1.3]{cheng2}). The stack $\mathcal{M}^{\epsilon}_{n,d}$ is thus a $\mu_2$-gerbe over the affine variety $ M^{\epsilon}_{n,d}$.

\begin{oss}
\label{cohomology}

The canonical morphism $q:\mathcal{M}^{\epsilon}_{n,d} \rightarrow M^{\epsilon}_{n,d}$, being a $\mu_2$-gerbe, is proper. The proper base change for Artin stacks implies that for every $x \in M^{\epsilon}_{n,d}$ and for every $i \in \Z$ we have $$(R^{i}q_*\C)_x=H^{i}(B(\mu_2),\C) .$$ As the rational higher cohomology of $B(\mu_2)$ vanishes, $R^{i}q_*\C=0$ if $ i \neq 0$ and $q_*\C=\C .$ The Leray spectral sequence for cohomology with compact support implies that $H^p_c(\mathcal{M}^{\epsilon}_{n,d}) \cong H^p_c(M^{\epsilon}_{n,d}) .$ The cohomology of the quotient stack is isomorphic to that of the GIT quotient: in particular, the (compactly-supported) cohomology of  $\mathcal{M}^{\epsilon}_{n,d}$ is $0$ in negative degrees.
\end{oss}

\vspace{10 pt}

\vspace{10 pt}

 The main result of this paragraph is the following proposition:

\begin{prop}
\label{connc}
\noindent

(i) If $r$ is odd there are no quaternionic  representations inside  $M_{n,d}^{\sigma}$ . If $r$ is even, $M_{n,d}^{\sigma}$ admits a  decomposition into open-closed subvarieties  $$M_{n,d}^{\sigma}=M_{n,d}^{\sigma,+} \bigsqcup M_{n,d}^{\sigma,-} $$ where $M_{n,d}^{\sigma,+},M_{n,d}^{\sigma,-}$ are given by real/quaternionic representations respectively and there is an isomorphism $M_{n,d}^{\sigma,+} \cong M_{n,d}^{\sigma,-}$. 

\noindent
(ii)  
The map $f:M^{\epsilon}_{n,d} \to M_{n,d}^{\sigma,+}$   introduced in Remark \ref{map} is an isomorphism.

\end{prop}

\vspace{10 pt}

 Before proving  Proposition \ref{connc}, we notice  that the quaternionic and the real representations form disjoint subsets by Remark \ref{rqt}. To see that there are no quaternionic representations for $r$ odd, we will use the  equivalence between quaternionic representations and quaternionic Higgs bundles. As this correspondence is crucial for the study of the varieties $M_{n,d}^{\epsilon}$, let us briefly review it here. For more details, see for example \cite{GP},\cite{FS2},\cite{ba1},\cite{ba2},\cite{BGP}

\subsubsection{Real and quaternionic Higgs bundles}
\label{Dol}
A Higgs bundle over $\Tilde{\Sigma}$ is a couple $(\mathcal{E},\Phi) $ where $\mathcal{E}$ is a vector bundle over $\Tilde{\Sigma}$ and $\Phi$ a morphism $\Phi:\mathcal{E} \to \mathcal{E} \otimes \Omega^1_{\tilde{\Sigma}}$. The moduli space of (stable) Higgs bundle over $\Tilde{\Sigma}$ of rank $n$ and degree $d$ is denoted by $M_{Dol,n,d}$ (for a definition of stability see for example \cite[Section 4.1]{ba1} or \cite[Definition 2.3]{GP}). It is a fundamental result (see for example \cite{Simp}) that there is a homeomoprhism (called non abelian Hodge correspondence)
\begin{equation}
\label{homeo}
M_{Dol,n,d} \cong M_{n,d}. \end{equation} 

We consider the involution on $M_{Dol,n,d}$ , which we denote   again by 
 $\sigma$, given by $$\sigma((\mathcal{E},\Phi))=(\sigma^*(\overline{\mathcal{E}}),-\sigma^*(\overline{\Phi})) $$ and we say that a Higgs bundle $(\mathcal{E},\Phi)$ is $\sigma$-invariant if there exists an isomorphism $\alpha:(\mathcal{E},\Phi) \to \sigma((\mathcal{E},\Phi))$. Real Higgs bundles are couples $((\mathcal{E},\Phi),\alpha)$ such that $$\sigma^*(\overline{\alpha})\alpha=I_{\mathcal{E}} .$$  In a similar way,  quaternionic Higgs bundles are defined by asking for the equality $$\sigma^*(\overline{\alpha})\alpha=-I_{\mathcal{E}} .$$

In \cite[Proposition 5.6]{GP},\cite[Theorem 4.8]{BGP} it is shown that the homemorphism  (\ref{homeo}) restricts to a homeomorphism $M_{n,d}^{\sigma} \cong M_{Dol,n,d}^{\sigma}$. In \textit{loc.cit} it is  shown moreover that this bijection  sends real/quaternionic representations into real/quaternionic Higgs bundles respectively. We will denote the subsets of $M_{Dol,n,d}$ given by real/quaternionic Higgs bundles by $M_{Dol,n,d}^{\sigma,+}/M_{Dol,n,d}^{\sigma,-} $ respectively. 

Notice that, as $\sigma:M_{Dol,n,d}\to M_{Dol,n,d}$ is antiholomorphic, the fixed points locus $M_{Dol,n,d}^{\sigma}$ is not a complex algebraic variety anymore but is identified with the set of $\mathbb{R}$-points of $M_{Dol,n,d}$ with respect to the real structure induced by $\sigma$.

\vspace{12 pt}
 For odd $r$, if a quaternionic couple $(\tilde{\rho},h) $ existed (i.e $M_{n,d}^{\sigma,-} \neq \emptyset$) 
 there would exist a stable quaternionic Higgs bundle $(\mathcal{E},\Phi)$ on $\Tilde{\Sigma}$. Its determinant $\det( \mathcal{E}) $ would be a quaternionic line bundle of degree $d$ over $\Tilde{\Sigma}$: the quaternionic condition is preserved under taking the determinant as $n$ is odd. The existence of a quaternionic line bundle for odd $r$ is ruled out by the topological criterion of \cite[Theorem 2.4]{FS}.  
\vspace{10 pt}

To prove Proposition \ref{connc}, we will need the following preliminary Lemma.

\begin{lemma}
\label{princ}
Put $\Hom_{n,d}(\Tilde{\Pi},G) \coloneqq \Hom_{\tilde{\mathcal{C}}}(\Tilde{\Pi},G)$ and let us consider the varieties $Y,Z$ defined by $$Y \coloneqq \Hom_{n,d}(\Tilde{\Pi},G) \times_{M_{n,d}}\Hom_{n,d}(\Tilde{\Pi},G)=\{(\tilde{\rho}_1,\tilde{\rho}_2) \ | \ \tilde{\rho}_1 \cong \tilde{\rho}_2 \} $$ and $$Z \coloneqq \{(\tilde{\rho}_1,\tilde{\rho}_2,h) \ | \ (\tilde{\rho}_1,\tilde{\rho}_2) \in Y \ , \ h \in \Gl(n) \ | \ h\tilde{\rho}_1h^{-1}=\tilde{\rho}_2 \} .$$

The projection map $\psi:Z \to Y$ is a principal $\mathbb{G}_m$-bundle for the \'etale topology.
 \end{lemma}

\begin{proof}
The variety $Z$ is endowed with the $\mathbb{G}_m$ action $t \cdot (\rho_1,\rho_2,h)=(\rho_1,\rho_2,th) .$ This action is free and transitive on the fibers of $\psi$, as all the representations inside $\Hom_{n,d}(\Tilde{\Pi},G)$ are irreducible. Moreover $\psi(t \cdot z)=\psi(z)$ for all $z \in Z$. We are thus reduced to show that $\psi$ is locally trivial for the \'etale topology. 

As the map $\tilde{q}:\Hom_{n,d}(\tilde{\Pi},G) \to M_{n,d}$ is a principal $\PGl_n$-bundle for the \'etale topology, there exists an \'etale open covering  $\{U_i\}_{i \in I}$ of $ M_{n,d}$ such that $\tilde{q}^{-1}(U_i) \cong U_i \times \PGl_n$ for each $i \in I$. Put $Y_{U_i}\coloneqq Y \times_{M_{n,d}} U_i$ and similarly $Z_{U_i} \coloneqq Z \times_{M_{n,d}} U_i$. It is enough to show that the pullback map $\psi:Z_{U_i} \to Y_{U_i}$ is locally trivial in the \'etale topology for each $i \in I$. 

Fix then  $i \in I$ and put $U_i=U$. Notice that the variety $Y_{U}$ admits the following isomorphism:  $$Y_U=\tilde{q}^{-1}(U)\times_U \tilde{q}^{-1}(U) \cong (U \times \PGl_n) \times_U (U \times \PGl_n) \cong U \times \PGl_n \times \PGl_n .$$

 In a similar way, the variety $Z_U$ is isomorphic to    $$Z_U=\psi^{-1}(Y_U)=\{(u,g,h,s) \in U \times \PGl_n \times \PGl_n \times \Gl_n \ | \ gh^{-1}=[s] \}$$ so that $\psi$ corresponds to the morphism $\psi(u,g,h,s)=(u,g,h)$. We can view $Y_U$ as a subset of $ U \times \PGl_n \times \PGl_n \times \PGl_n$ as $$Y_U=\{(u,g,h,s) \in  U \times \PGl_n \times \PGl_n \times \PGl_n \ | \ gh^{-1}=s \} .$$ 
 
 Via these identifications, the map $\psi$ corresponds to the restriction of the morphism\\ $ U \times \PGl_n \times \PGl_n \times \Gl_n \to  U \times \PGl_n \times \PGl_n \times \PGl_n $ given by the identity on the first three factors and the  quotient map $\Gl_n \to \PGl_n$ on the last one. This is a principal $\mathbb{G}_m$-bundle because $\Gl_n \to \PGl_n$ is so.

\end{proof}

\vspace{10 pt}

\vspace{10 pt}

We now prove  Proposition \ref{connc}. We keep the notations of Lemma \ref{princ}.

\begin{proof}[Proof  of Proposition \ref{connc}]
  $ $\newline
 Put $\Hom_{n,d}(\tilde{\Pi},G)^{\sigma}=\tilde{q}^{-1}(M_{n,d}^{\sigma})$ and  $\Hom_{n,d}(\tilde{\Pi},G)^{\sigma,+}=\tilde{q}^{-1}(M_{n,d}^{\sigma,+})$  and similarly for quaternionic representations $\Hom_{n,d}(\tilde{\Pi},G)^{\sigma,-}=\tilde{q}^{-1}(M_{n,d}^{\sigma,-})$. The variety $\Hom_{n,d}(\tilde{\Pi},G)^{\sigma} $ is isomorphic to the closed subvariety $Y^{\sigma}$ of $Y$ given by: $$Y^{\sigma}=\{(\tilde{\rho}_1,\tilde{\rho}_2) \in Y \ | \ \tilde{\rho}_2=\theta \tilde{\rho}_1 \sigma_* \} $$ via the map $p_1|_{Y^{\sigma}}:Y^{\sigma}\to \Hom_{n,d}(\tilde{\Pi},G)^{\sigma} $, where $p_1$ is the projection onto the first factor of $Y$. Put $Y^{\sigma,+}=p_1^{-1}(\Hom_{n,d}(\tilde{\Pi},G)^{\sigma,+})$ and similarly $Y^{\sigma,-}$. From Remark \ref{rqt} there is a well-defined  morphism  $p_3:\psi^{-1}(Y^{\sigma}) \to \{I_n,-I_n\}$ $$(\tilde{\rho}_1,\tilde{\rho}_2,h) \to \theta(h)h\tilde{\rho}(\sigma(\lambda_{\sigma})\lambda_{\sigma}) .$$
 
Notice that  $Y^{\sigma,+}=\psi(p_3^{-1}(I_n))$ and $Y^{\sigma,-}=\psi(p_3^{-1}(-I_n))$. As $\psi$ is open, we deduce that $\Hom_{n,d}(\tilde{\Pi},G)^{\sigma,+},$
 $\Hom_{n,d}(\tilde{\Pi},G)^{\sigma,-}$ are disjoint and open and so closed too inside $\Hom_{n,d}(\tilde{\Pi},G)^{\sigma}$. The same is true then for $M_{n,d}^{\sigma,+},M_{n,d}^{\sigma,-}$. The projection $(\tilde{\rho}_1,\tilde{\rho}_2,h) \to (\tilde{\rho}_1,h)$ induces an isomorphism $\psi^{-1}(Y^{\sigma,+})=\Hom_{n,d}(\Pi,G^+)^{\epsilon}$. By Proposition \ref{princ}, the morphism $$\Hom_{n,d}(\Pi,G^+)^{\epsilon} \to\Hom_{n,d}(\tilde{\Pi},G)^{\sigma,+} $$ is thus a principal $\mathbb{G}_m$-bundle. The $G$-action on $\Hom_{n,d}(\Pi,G^+)^{\epsilon}$  defined by the Formula (\ref{action}) induces an action of the center $Z_G=\mathbb{G}_m$ which differs from the one coming from the principal $\mathbb{G}_m$-bundle structure by a square factor. The morphism $\Hom_{n,d}^{\epsilon}(\Pi,G^+) \to \Hom_{n,d}(\tilde{\Pi},G)^{\sigma,+}$ induces thus a $G$-equivariant isomorphism \begin{equation}
    \label{isom2}
    \Hom_{n,d}(\Pi,G^+)^{\epsilon}/(\mathbb{G}_m/(\pm I_n)) \cong \Hom_{n,d}(\tilde{\Pi},G)^{\sigma,+} 
\end{equation}

We deduce the following chain of isomorphisms: $$M^{\epsilon}_{n,d}=\Hom_{n,d}(\Pi,G^+)^{\epsilon}/(G/(\pm I_n)) \cong  (\Hom_{n,d}(\Pi,G^+)^{\epsilon}/(\mathbb{G}_m/(\pm I_n)))/(G/\mathbb{G}_m) \cong M^{\sigma,+}_{n,d} .$$
\end{proof}

To end the proof of Proposition \ref{connc}, it actually remains to show that  $M_{n,d}^{\sigma,+},M_{n,d}^{\sigma,-}$ are  isomorphic if $r$ is even. For $r$ even there exists a quaternionic representation $\tau \in M_{1,0}^{\sigma,-}$ of rank $1$ over $\Tilde{\Sigma}$ (see  \cite[Theorem 2.4]{FS2}). Taking the tensor product by $\tau$ gives then an isomorphism $- \otimes \tau:M_{n,d}^{\sigma,+} \to M_{n,d}^{\sigma,-}$: the same proof was carried out for real and quaternionic vector bundles in \cite[Theorem 1.1]{FS2}.

\subsection{Character stacks for (real) elliptic curves}

\label{mainresult}

We focus now on the case $r=2$. We consider the elliptic curve $\Tilde{E}$ and the antiholomorphic involution $\sigma$ introduced in Example \ref{ellip}. We keep the notations introduced in the  Example \ref{ellip}. 

 In \cite[Lemma 2.2.6]{HRV} it is shown that for $(n,d)=1$ there is an isomorphism  \begin{equation}
 \label{ziopera}
 M_{n,d}=\C^* \times \C^* .\end{equation}
  To see this, notice that a representation $\tilde{\rho} \in M_{n,d}$ corresponds to a pair of matrices $A,B$ such that $$B^{-1}A^{-1}BA=e^{\frac{2\pi i d}{n}}1_n .$$ where $\tilde{\rho}(a)=A$ and $\tilde{\rho}(b)=B$. Let $z,w \in \C^*$ such that $A^n=z I_n $ and $B^n=wI_n$(see \cite[Theorem 2.2.17]{HRV}). The isomorphism  (\ref{ziopera}) is obtained by mapping $\tilde{\rho}$ to the couple $(z,w)$. Via this identification, the involution $\sigma$ is given by:  $$\sigma(z,w)=(z,w^{-1}) $$ and so $M_{n,d}^{\sigma}=\C^* \bigsqcup \C^* .$ From Equation (\ref{eqell1}) we deduce indeed that $$\theta\tilde{\rho}(\sigma_*(b))=\theta(\tilde{\rho}(b))=\theta(B)  $$ and so $(\theta\tilde{\rho}(\sigma_*(b)))^n=\theta(B)^n=w^{-1}I_n$. By Equation (\ref{eqell2}) the following equality holds: $$\theta\tilde{\rho}(\sigma_*(a))=\theta(\tilde{\rho}(x_1a^{-1}))=\theta(\tilde{\rho}(x_1))\theta(A^{-1})=e^{-\frac{\pi d}{n}}A^t $$ and so $(\theta(\tilde{\rho}(\sigma_*(a)))^n=(A^t)^n=zI_n .$ By Proposition \ref{connc}, we deduce the following result :
 \begin{teorema}
 \label{ellip1}
 For $r=2$, the character variety $M_{n,d}^{\epsilon}$ is isomorphic to $\C^*$ as an affine variety and  the character stack $\mathcal{M}_{n,d}^{\epsilon} $ is a $\mu_2$-gerbe over $\C^*$. 
\end{teorema}

 By Remark \ref{cohomology}, for $r=2$  the following identity holds: \begin{equation}H_c(\mathcal{M}_{n,d}^{\epsilon},q,t)=qt^2+t. \end{equation}

As suggested in the introduction, this does not agree with the expected formula (\ref{falseconj}). If the Formula (\ref{falseconj}) were true, the following identity would hold $$H_c(\mathcal{M}_{n,d}^{\epsilon},q,t)=\dfrac{qt^2}{qt^2-1}\mathbb{H}_{n,2}\left(t\sqrt{q},-\dfrac{1}{\sqrt{q}}\right)$$ where $\mathbb{H}_{n,2}(z,w)$ are the functions defined by Equation (\ref{formula}) for $\mu=((n))$. The functions $\mathbb{H}_{n,2}(z,w)$ have been explicitly computed  in \cite[Theorem 1.0.2]{Carlsson}. The result of \cite{Carlsson} agrees with the conjectural formula (\ref{mhp}) for the mixed Poincar\'e series of character varieties $\mathcal{M}_{n,d}$ for elliptic curves, i.e $(qt^2)\mathbb{H}_{n,2}\left(t\sqrt{q},-\dfrac{1}{\sqrt{q}}\right)=(qt^2+t)^2$. This implies that \begin{equation}\dfrac{(qt^2+t)^2}{qt^2-1}=\dfrac{qt^2}{qt^2-1}\mathbb{H}_{n,2}\left(t\sqrt{q},-\dfrac{1}{\sqrt{q}}\right) \neq (qt^2+t) \end{equation} giving a counterexample to the  conjectural formula (\ref{falseconj}).

\end{document}